\documentclass[b5paper,reqno]{amsart}
\usepackage{amsfonts,lingmacros,tree-dvips, amsmath, amssymb, graphicx, mathrsfs}
\usepackage[hidelinks]{hyperref}
\usepackage{enumerate}
\usepackage[all]{xy}

\newdir{ >}{!/6pt/@{ }*:(1,0.2)@_{>}*:(1,-0.2)@^{>}}

\theoremstyle{plain}

\newtheorem{theorem}{Theorem}[section]

\theoremstyle{definition}
\newtheorem{definition}[theorem]{Definition}

\newtheorem{counter example}[theorem]{Counter Example}

\numberwithin{equation}{section}

\DeclareMathAlphabet{\mathscr}{OT1}{pzc}{m}{it} 
\begin{document}
	\Large{
		\title{ON A PROBLEM OF KURATOWSKI}
		
		\author[S. Basu]{Sanjib Basu}
		\address{\large{Department of Mathematics,Bethune College,181 Bidhan Sarani}}
		\email{\large{sanjibbasu08@gmail.com}}
		
		\author[A.C.Pramanik]{Abhit Chandra Pramanik}
		\address{\large{Department of Pure Mathematics, University of Calcutta, 35, Ballygunge Circular Road, Kolkata 700019, West Bengal, India}}
		\email{\large{abhit.pramanik@gmail.com}}

		\thanks{The second author thanks the CSIR, New Delhi – 110001, India, for financial support}
		\begin{abstract}
			If $f:X\mapsto Y$ is a function having Baire property from a metric space $X$ into a separable metric space $Y$, then $f$ is continuous except on a set of first category. Kuratowski asked whether the condition of separability could be removed. Several attempts were done in the past to solve this problem. In fact, the first impressive attempt was initiated by Kunugi. This paper is aimed towards solving the problem of Kuratowski in category bases which generalizes the result of Kunugi.  
		\end{abstract}
		\thanks{}
		\maketitle
		
		\section{INTRODUCTION}
		Kuratowski [5], using continuum hypothesis, proved the following 
		\begin{theorem}
			If $\mathcal{F}$, $|\mathcal{F}|\leq\omega_1$ is a partition  of [0,1] into meager sets, then there exists a family $\mathcal{A}\subseteq\mathcal{F}$ such that $\bigcup\mathcal{A}$ has not the Baire property.
		\end{theorem} 
	 The measure analogue of the above result was known to Solovay who established the same using generic ultrafilter. Bukovsky [1] using forcing argument gave a much shorter and less complicated proof of Solovay's result. However, none of them used continuum hypothesis. Using measure category duality, one can easily transcibe the proof of Bukovsky's theorem into its dual statement which is Theorem 1.1.\\
	 
	 A generalization of Theorem 1.1 appeared in [2]. There Emeryk, Frankiewicz and Kulpa proved that $-$	  
		\begin{theorem}
			Let $X$ be a $\check{C}$ech complete space of weight $\leq 2^\omega$ and $\mathcal{F}$ be a partition of $X$ into meager sets. Then there exists a subfamily $\mathcal{A}\subseteq\mathcal{F}$ such that $\bigcup\mathcal{A}$ has not the Baire property.
		\end{theorem} 
		
		The above attempt to establish an improved version of Kuratowski's result was directed towards solving a 1935 problem of Kuratowski. It is already known that if $X$ is a metric space, $Y$ is a separable metric space and $f:X\mapsto Y$ is a function having Baire property, then $f$ is continuous on the complement of a set of first category. Kuratowski [6]  asked whether the same is true if $Y$ is not separable. Using the above theorem, Emeryk, Frankiewicz and Kulpa [3] attempted to answer this question. 
		\begin{theorem}
			Let $X$ be a $\check{C}$ech complete metric space of weight less than or equal to $2^{\aleph_0}$. If $f:X\mapsto Y$ is a function having Baire property from $X$ into a metric space $Y$, then $f$ is continuous on the complement of a set of first category.
		\end{theorem}
		
		Only recently, it came to light (thanks to the work of Grzegorek and Labuda with historical overview) that an attempt to solve the Kuratowski problem originated much earlier in a paper of 1936 by Kinjiro Kunugi [7] and was based on a theorem of Luzin and Novikov. Kunugi imposed the following restriction on the space $X$$-$\\
		
		`` If given an ordinal number $\gamma$, \{$A_\gamma\}_{\xi<\gamma}$ is a family of mutually disjoint first category sets such that their union $\bigcup\limits_{\xi<\gamma}A_\xi$ is a set of second category, then there exist an open set $O$ and two disjoint unions $\bigcup\limits_{\xi'}A_{\xi'}$ and $\bigcup\limits_{\xi''}A_{\xi''}$ of subfamilies each of which is everywhere of second category in $O$ "\\
		
		and proved that $-$
		\begin{theorem}
			If $X$ is a topological space satisfying the above restriction and $f:X\mapsto Y$ is a function having Baire property from $X$ into a metric space $Y$, then $f$ is continuous except on a set of first category in $X$.
		\end{theorem}
		In the theorem stated below, Grzegorek and Labuda [4] showed that the relevant topological fact which implies the validity of Kunugi's assumption is that the space should have a countable base of open sets, or, in otherwords, should be second coutable.
		\begin{theorem}
			Let \{$A_\gamma:\gamma\in\Gamma$\} ($\Gamma$ is an index set) be a family of mutually disjoint first category subsets of a second countable topological space $X$. If the union $E=\bigcup\limits_{\gamma\in\Gamma}A_\gamma$ is a set of second category, then there exist $\Delta\subseteq\Gamma$ and an open set $O$ such that both $\bigcup\limits_{\gamma\in\Delta}A_\gamma$ and $\bigcup\limits_{\gamma\in{\Gamma\setminus\Delta}}A_\gamma$ are everywhere of second category in $O$. As a consequence, $\bigcup\limits_{\gamma\in\Delta}A_\gamma$ and $\bigcup\limits_{\gamma\in{\Gamma\setminus\Delta}}A_\gamma$ cannot be separated by sets having Baire property. 
		\end{theorem}
	In the same paper, they gave an updated proof of theorem 1.4 keeping as close as possible to the original proof of Kunugi. This paper attempts to solve Kuratowski's problem in category bases which generalizes Kunugi's theorem.
		
			\section{PRELIMINARIES AND RESULTS}
		The concept of  category base is a generalization of both measure and topology. Its main objective is to present both measure and Baire category (topology) and some other aspects of point set classification within a common framework. It was introduced by J.C.Morgan II [8] in the seventies of the last century and since then has been developed through a series of papers [9], [10], [11], [12] etc.\\
		
		The relevant definitions (needed for our purpose) are as follows $-$
		
		\begin{definition}
			A pair $(X,\mathcal{C}$) where $X$ is a non-empty set and $\mathcal{C}$ is a family of subsets of $X$ is called a category base if the non-empty members of $\mathcal{C}$ called regions satisfy the following set of axioms :
			\begin{enumerate}
				\item Every point of $X$ is contained in at least one region; i.e., $X = \cup$ $\mathcal{C}$.
				\item Let $A$ be a region and $\mathcal{D}$ be any non-empty family of disjont regions having cardinality less than the cardinality of $\mathcal{C}$.\\
				i) If $A \cap ( \cup \mathcal{D}$) contains a region, then there exists a region $D\in\mathcal{D}$ such that $A\cap D$ contains a region. \\
				ii)  If $A\cap(\cup \mathcal{D})$ contains no region, then there exists a region $B\subseteq A$ which is disjoint from every region in $\mathcal{D}$.
			\end{enumerate}
		\end{definition}
	Several examples of category bases may be found in the book of Morgan and also in the papers [9], [10].
		\begin{definition}
			$[8]$ In a category base ($X,\mathcal{C}$), a set is called `singular' if every region contains a region which is disjoint from the set itself. A set which can be expressed as a countable union of singluar sets is called `meager'. Otherwise, it is called `abundant'. A set is called `Baire' if every region contains a subregion in which either the set or its complement is meager.
		\end{definition}
		A category base in which every region is abundant, is called a `Baire Base'.
		\begin{definition}
			$[8]$ In a category base ($X,\mathcal{C}$), a subfamily $\mathcal{B}$ of regions is called a `pseudobase' if every region contains a region from $\mathcal{B}$. In fact, every category base is a pseudobase of itself.
		\end{definition}
	\newpage
		The following theorem generalizes Theorem 1.5.
		\begin{theorem}
			Let ($X,\mathcal{C}$) be a Baire base satisfying the following property\\ $(\star)$ for every region $A$, the induced category base ($A,\mathcal{C}_{|A}$) has a countable pseudobase.\\
			Then for every family \{$A_{\alpha}:\alpha\in\Lambda$\} of disjoint meager sets whose unoin is abundant in ($X,\mathcal{C}$), there exist two subfamilies \{$A_\alpha:\alpha\in\Lambda_1$\} and \{$A_\alpha:\alpha\in\Lambda_2$\} whose union is $\bigcup\limits_{\alpha\in\Lambda}A_\alpha$ and a region $C$ such that both $\bigcup\limits_{\alpha\in\Lambda_1}A_\alpha$ and $\bigcup\limits_{\alpha\in\Lambda_2}A_\alpha$ are abundant everywhere in $C$.
		\end{theorem}
	\begin{proof}
		Let $\Omega$ be the least ordinal representing the cardinality $|\Lambda|$ of $\Lambda$. Without loss in generality, we may assume that $\Omega$ is also the least ordinal for which $\bigcup\limits_{\alpha<\Omega}A_{\alpha}$ is abundant.\\
		Now let $B_\alpha=\bigcup\limits_{\beta<\alpha}A_\beta=\bigcup\limits_{n=1}^{\infty}B_{\alpha}^{(n)}$ ($\alpha<\Omega$), where $B_{\alpha}^{(n)}$ are singular sets in $X$ (this way of representing of $B_\alpha$ is justified according to the choice of the ordinal $\Omega$). For any $y\in\bigcup\limits_{\alpha<\Omega}A_\alpha$, one can find $\alpha$ such that $y\in B_\beta$ for every $\beta>\alpha$. If $\gamma$ is the least ordinal for which $y\in B_\gamma$, then $\bigcup\limits_{n=1}^{\infty}\mathcal{U}_{n,y}=\bigcup\limits_{\alpha<\Omega}A_\alpha-\bigcup\limits_{\alpha<\gamma}A_\alpha$ where $\mathcal{U}_{n,y}=\bigcup\{A_\alpha:y\in B_{\alpha}^{(n)}\}$ which is abundant in every region in which $\bigcup\limits_{\alpha<\Omega}A_\alpha$ is abundant.\\
			By the Fundamental theorem (Ch 1, II, F, [8]), $\bigcup\limits_{\alpha<\Omega}A_\alpha$ is abundant everywhere in some region $D$ (say). Consequently, for every $y\in\bigcup\limits_{\alpha<\Omega}A_\alpha$, $\bigcup\limits_{n=1}^{\infty}\mathcal{U}_{n,y}$ is abundant everywhere in $D$ and if \{$D_n:n\in\mathbb{N}$\} ($\mathbb{N}$ is the set of natural numbers) constitutes a countable pseudobase for ($D,\mathcal{C}_{|D}$), then for every $y$ there is a pair ($n_y,m_y)\in\mathbb{N}\times\mathbb{N}$ such that $\mathcal{U}_{n_y,y}$ is abundant everywhere in $D_{m_y}$. We set $T_{m,n}=\{y\in\bigcup\limits_{\alpha<\Omega}A_\alpha: \mathcal{U}_{n_y,y}$ is abundant everywhere in  $D_{m_y}; n_y=n, m_y=m $\} so that $\bigcup\limits_{m,n=1}^{\infty}T_{m,n}$ is abundant everywhere in $D$. Consequently, there exists a subregion $P_{m_0}$ of $D_{m_0}$ such that $T_{n_0,m_0}$ is abundant everywhere in $P_{m_0}$. In otherwords, $T_{n_0,m_0}=\{y\in\bigcup\limits_{\alpha<\Omega}A_\alpha: \bigcup\{A_\alpha: y\in B_{\alpha}^{(n_0)}\}$ is abundant everywhere in $D_{m_0}$\} is abundant everywhere in $P_{m_0}$.\\
			
			Let \{$P_{k}^{(m_0)}: k\in\mathbb{N}$\} be a countable pseudobase for ($P_{m_0},\mathcal{C}_{| P_{m_0}}$). We set $M_k=\bigcup\{A_\alpha:B_{\alpha}^{(n_0)}\cap P_{k}^{(m_0)}=\phi$\}. Since $\bigcup\limits_{k=1}^{\infty}M_k=\bigcup\limits_{\alpha<\Omega}A_\alpha$, there exists $k_0\in\mathbb{N}$ such that $M_{k_0}$ is abundant everywhere in $P_{k_0}^{(m_0)}$. Choose $y\in T_{n_0,m_0}\cap P_{k_0}^{m_0}$ and label it as $y_0$. Then $\bigcup\{A_\alpha:y_0\in B_{\alpha}^{(n_0)}\}=\mathcal{U}_{n_0,y_0}$ and $M_{k_0}$ which are obviously disjoint unions are both abundant in $P_{k_0}^{(m_0)}$.\\
			From the above deductions, we finally conclude that there exists a region $C$ and a pair ($\Lambda_1,\Lambda_2$) such that $\Lambda_1\cup\Lambda_2=\Lambda$, $\Lambda_1\cap\Lambda_2=\phi$ and both $\bigcup\limits_{\alpha\in\Lambda_1}A_\alpha$ and $\bigcup\limits_{\alpha\in\Lambda_2}A_\alpha$ are abundant everywhere in $C$.\\
			This proves the theorem.
	\end{proof}
An easy consequence of the above theorem is this that none of the unions $\bigcup\limits_{\alpha\in\Lambda_1}A_{\alpha}$ and $\bigcup\limits_{\alpha\in\Lambda_2}A_{\alpha}$ are Baire sets in the category base ($X,\mathcal{C}$).\\

As a part of our attempt to solve Kuratowski's problem, we now embark on our formulation of Kunugi's theorem in the general settings of category bases. To start with, we assume that the co-domain space $Y$ of our function $f:X\mapsto Y$ is a topological space satisfying the following condition : \\
($\star\star$) $Y=\bigcup\limits_{n,k=1}^{\infty}D_{n,k}$ (where $D_{n,k}=\bigcup\limits_{\xi}D_{n,k}^{(\xi)}$) is a disjoint union of Borel sets such that 
\begin{enumerate}[(i)]
	\item The union $\bigcup\limits_{\xi'}D_{n,k}^{(\xi')}$ of any subfamily is also a Borel set .
	\item Every nonempty open subset $G$ of $Y$ contains some $D_{n,k}^{(\xi)}$, so that $G$ can be expressed as union of subfamilies \{$D_{n,k}^{(\xi')}$\} of \{$D_{n,k}^{(\xi)}$\} for which $D_{n,k}^{(\xi')}\subseteq G$. More precisely, there exists a set $\Sigma\subseteq\mathbb{N}\times\mathbb{N}$ such that $G=\bigcup\limits_{(n,k)\in\Sigma}\bigcup\limits_{\xi'}D_{n,k}^{(\xi')}$
\end{enumerate}
By virtue of Montegomery lemma (in the form it is used in [4]), any non-separable metric space is a topological space of the above type.
		
\begin{theorem}
	Let ($X,\mathcal{C}$) be a Baire base where nonempty intersection of any two region is not singular and satisfying condition ($\star$) stated in theorem 2.4 and $f:X\mapsto Y$  is a Baire function from $X$ into  a topological space $Y$ satisfying condition ($\star$$\star$) stated above. Then there exists a meager set $P$ such that the restriction $f_{|X-P}$ has the following property : For every nonempty open subset $H$ of $Y$, $f_{|X-P}^{-1}(H)=\bigcup\limits_{n=1}^{\infty}(\bigcup\mathcal{C}_{H}^{(n)})-P$, where for each $n\in\mathbb{N}$, $\bigcup\mathcal{C}_{H}^{(n)}$ is a disjoint union of regions.	
\end{theorem}

		\begin{proof}
			For each pair $(n,k)\in\mathbb{N}\times\mathbb{N}$, since $D_{n,k}=\bigcup\limits_{\xi}D_{n,k}^{(\xi)}$ is a union of Borel sets, so $f^{-1}(D_{n,k})=\bigcup\limits_{\xi}f^{-1}(D_{n,k}^{(\xi)})=\bigcup\limits_{\xi}(M_{n,k}^{(\xi)}-F_{n,k}^{(\xi)})\cup T_{n,k}^{(\xi)}$, where $M_{n,k}^{(\xi)}$ are regions and $F_{n,k}^{(\xi)}$, $T_{n,k}^{(\xi)}$ ($F_{n,k}^{(\xi)}\subseteq M_{n,k}^{(\xi)}; F_{n,k}^{(\xi)}\cap T_{n,k}^{(\xi)}=\phi$) are meager sets (this presentation is a part of the proof of Th 7, III, E, Ch 1, [8]). Moreover, as $D_{n,k}^{(\xi_1)}\cap D_{n,k}^{(\xi_2)}=\phi$ for $\xi_1\neq\xi_2$ from the hypothesis and the fact that ($X,\mathcal{C}$) is a Baire base, it follows that $M_{n,k}^{(\xi_1)}\cap M_{n,k}^{(\xi_2)}=\phi$.\\
			We now check that $\bigcup\limits_\xi F_{n,k}^{(\xi)}$ is also meager. For this we need only establish that if \{$C_\alpha\}_{\alpha\in\Lambda}$ is a family of disjoint regions and $B_\alpha(\subseteq C_\alpha)$ are meager, then $\bigcup\limits_{\alpha\in\Lambda}B_\alpha$ is also meager. But this is a consequence of the fact that $B_\alpha=\bigcup\limits_{n=1}^{\infty}B_{\alpha}^{(n)}$ where $B_{\alpha}^{(n)}$ are singular and therefore $\bigcup\limits_{\alpha\in\Lambda}B_{\alpha}^{(n)}$ is singular for each $n$ which may be derived from the definition of a category base.\\
			Now $\bigcup\limits_{\xi}f^{-1}(D_{n,k}^{(\xi)})=\bigcup\limits_{\xi}(M_{n,k}^{(\xi)}-F_{n,k}^{(\xi)})\cup T_{n,k}^{(\xi)}$ which by (i) of ($\star\star$) and the definition of Baire function can be expressed as ($M_{n,k}-F_{n,k})\cup T_{n,k}$ where $M_{n,k}$ is the union of disjoint regions and $F_{n,k}, T_{n,k}(F_{n,k}\subseteq M_{n,k}; T_{n,k}\cap M_{n,k}=\phi$) are both meager. Further, if we set $L_{n,k}=\bigcup\limits_{\xi}M_{n,k}^{(\xi)}$, then $\bigcup\limits_\xi T_{n,k}^{(\xi)}\subseteq \bigcup\limits_\xi F_{n,k}^{(\xi)}\cup\{(M_{n,k}-F_{n,k})-L_{n,k}\}\cup T_{n,k}$. Let $S_{n,k}=(M_{n,k}-F_{n,k})-L_{n,k}$. In order to show that $\bigcup\limits_\xi T_{n,k}^{(\xi)}$ is meager, we need only to establish that $S_{n,k}$ is meager.\\
			If possible, let $S_{n,k}$ be abundant. Then $\bigcup\limits_\xi(S_{n,k}\cap T_{n,k}^{(\xi)})$ is also abundant expressed as the disjoint union of meager sets. So by the conclusion drawn at the end of the proof of theorem 2.4, there is a subfamily \{$T_{n,k}^{(\xi')}$\} of \{$T_{n,k}^{(\xi)}$\} such that $\bigcup\limits_{\xi'}(S_{n,k}\cap T_{n,k}^{(\xi')})$ is non-Baire. But this is impossible, beacause $\bigcup\limits_{\xi'}(S_{n,k}\cap T_{n,k}^{(\xi')})=S_{n,k}\cap(\bigcup\limits_{\xi'}T_{n,k}^{(\xi')})=S_{n,k}\cap(\bigcup\limits_{\xi'}\{(M_{n,k}^{(\xi')}-F_{n,k}^{(\xi')})\cup T_{n,k}^{(\xi')}\})=S_{n,k}\cap (\bigcup\limits_{\xi'}f^{-1}(D_{n,k}^{(\xi')}))=S_{n,k}\cup f^{-1}(\bigcup\limits_{\xi'}D_{n,k}^{\xi'})$ which is obviously a Baire set because by hypothesis $\bigcup\limits_{\xi'}D_{n,k}^{(\xi')}$ is Borel.\\
			Hence for each ($n,k$), ($\bigcup\limits_\xi F_{n,k}^{(\xi)})\cup(\bigcup\limits_\xi T_{n,k}^{(\xi)})$ is meager. Since by (ii) of ($\star$$\star$), given any open set $H$, there exists $\Sigma\subseteq\mathbb{N}\times\mathbb{N}$ such that $H=\bigcup\limits_{(n,k)\in\Sigma}\bigcup\limits_{\xi'}D_{n,k}^{(\xi')}$, therefore, upon setting $P=\bigcup\limits_{(n,k)\in\Sigma}\bigcup\limits_{\xi}(F_{n,k}^{(\xi)}\cup T_{n,k}^{(\xi)})$ which is meager, we may write $f_{|X-P}^{-1}(H)=(\bigcup\limits_{(n,k)\in\Sigma}\bigcup\limits_{\xi'}M_{n,k}^{\xi'})-P$ which may be finally expressed as $f_{|X-P}^{-1}(H)=\bigcup\limits_n(\bigcup \mathcal{C}_{H}^{(n)})-P$ where $\bigcup\limits_n$ is taken over a subset of $\mathbb{N}$ and for each $n$ in that union, the symbol $\bigcup\mathcal{C}_{H}^{(n)}$ stands for the union of a subfamily of disjoint regions.\\
			Hence the theorem.	
		\end{proof}
	
	Here it may be noted that if the category base is a second countable Baire topological space, then the hypothesis of Theorem 2.5 is obviously satisfied and in that case it follows from the conclusion of the theorem that any function (having Baire property from a second countable topological space into a metric space) is continuous except on a set of first category in $X$.

		\bibliographystyle{plain}

	\end{document}